\documentclass[11pt]{amsart}
\usepackage{pstricks,pst-plot}

\definecolor{Black}{cmyk}{0,0,0,1}
\definecolor{OrangeRed}{cmyk}{0,0.6,1,0} % half magenta only, full yellow
\definecolor{DarkBlue}{cmyk}{1,1,0,0.20}
\definecolor{myblue}{rgb}{0.66,0.78,1.00}
\definecolor{Violet}{cmyk}{0.79,0.88,0,0}
\definecolor{Lavender}{cmyk}{0,0.48,0,0}

\parskip=\smallskipamount

\newtheorem{theorem}{Theorem}[section]
\newtheorem{lemma}[theorem]{Lemma}

\theoremstyle{definition}

\newtheorem{conjecture}[theorem]{Conjecture}

\newcommand{\bea}{\begin{eqnarray*}}
\newcommand{\eea}{\end{eqnarray*}}

\numberwithin{equation}{section}

\begin{document}

\title[ A global approximation result by Al Taylor and the strong openness conjecture in $\mathbb{C}^n$]{ 
 A global approximation result by Al Taylor and the strong openness conjecture in $\mathbb{C}^n$}
%{ A counterexample for strong openness conjecture to $\mathbb{C}^n$}A counterexample for strong openness %conjecture to $\mathbb{C}^n$}
\author{John Erik Forn\ae ss}
\author{Jujie Wu}
\address{{John Erik Forn\ae ss}
\\{\it E-mail address:} johnefo@math.ntnu.no
\\{ Department of Mathematical Sciences, NTNU}\\{ Sentralbygg 2, Alfred Getz vei 1, 7034 Trondheim, Norway}}

\address{{Jujie Wu}
\\{\it E-mail address:} 99jujiewu@tongji.edu.cn
\\{ School of Mathematics and statistics, Henan University}\\{ Jinming Campus of Henan University, Jinming District, City of Kaifeng, Henan Provence. PR. China 475001}}

\date{}
\maketitle

\bigskip

\begin{abstract}
We improve a global approximation result by Al Taylor in $\mathbb C^n$ for holomorphic functions in weighted Hilbert spaces. The main tools are a variation of the theorem of H\"ormander  on weighted $L^2$-estimates for the $\overline{\partial}$-equation together with the solution of the strong openness conjecture. A counterexample to a global strong openness conjecture in $\mathbb{C}^n $ is also given here.
\bigskip


\smallskip

\end{abstract}

\tableofcontents

\section{Introduction}
Let $z = (z_1, z_2, \cdots, z_n )\in \mathbb{C}^n$, $\|z\|^2 = |z_1|^2 +|z_2|^2 +\cdots + |z_n|^2$.  Let $\Omega \subset \mathbb{C}^n$ be a pseudoconvex domain.  If $\varphi $ is a measurable function in $\Omega$, we denote by $L^2(\Omega ,\varphi)$ the space of measurable functions $f$ in $\Omega $ which are square integrable with respect to the measure $e^{-\varphi}d\lambda$, i.e.,
$$
\|f\|_\varphi^2 : = \int_{\Omega } |f|^2 e^{-\varphi}d\lambda < + \infty,
$$
where $d\lambda$ is the Lebesgue measure. This is a subspace of the space $L^2(\Omega, \rm{loc})$ of functions in $\Omega$ which are locally square integrable with respect to the Lebesgue measure. By $L^2_{p,q} (\Omega , \varphi) $ we denote the space of differential forms of type $(p,q)$ with coefficients in $L^2 (\Omega, \varphi)$. By $H(\Omega,\varphi)$ we denote the set of holomorphic functions on $\Omega$ which belong to $L^2 (\Omega, \varphi)$.
%If $\Omega = \mathbb{C}^n$, we directly use $L^2(\varphi) $ and  $H(\varphi)$ to represent.
For a weight $\varphi, $ let $H(\varphi)$ denote the space of entire functions $f$ with finite $\varphi$ norm, i.e., $\|f\|^2_{H(\varphi)}=\int_{\mathbb C^n}|f|^2e^{-\varphi}d\lambda \le + \infty.$% convenience we denote $L^2(\varphi) = L^2  (\mathbb{C}^n,\varphi)$,  $L^2_{p,q} (\varphi) = L^2_{p,q} (\mathbb{C}^n, \varphi) $ and $H(\varphi)=H (\mathbb{C}^n,\varphi)$. \\

In 1971 Al Taylor [6] investigated weighted approximation results for entire functions in $\mathbb C^n$. He proved:

\begin{theorem}
Let $\varphi_1\le \varphi_2\le \varphi_3 \le \cdots$ be plurisubharmonic (psh) functions on $\mathbb C^n,$ let
$\varphi=\lim\limits_{j\rightarrow \infty} \varphi_j$, and suppose that
$\int_Ke^{-\varphi_1}d\lambda<\infty$ for every compact set $K.$ Then the closure of $\bigcup \limits _{j=1}^\infty H(\varphi_j+\log(1+\|z\|^2))$  in the Hilbert space $L^2( \varphi + \log(1+\|z\|^2))$ contains $H(\varphi)$.
\end{theorem}

The condition on the sequence of plurisubharmonic functions resemble the condition in the recent work on the strong openness conjecture.
Berndtsson proved the openness conjecture of Demailly-Koll\'ar as follows.

\begin{theorem}[Openness theorem]
Let $U$ be a bounded pseudoconvex domain and $\varphi \in psh^-(U)$ ($psh^-(U)$ means the set of negative psh function on $U$) with
$$
\int_U e^{-\varphi}d\lambda <+\infty.
$$
Let $V$ be a relatively compact domain in $U$. Then there exists $p>1$ such that
$$
\int_V e^{-p\varphi} d\lambda < +\infty.
$$
\end{theorem}

Slightly later,  Guan-Zhou (\cite{guanzhou1311} and \cite{guanzhou1401}) proved a strong openness conjecture theorem as follows.

\begin{theorem}\label{th:opennessconjecture}
Let $\varphi_1\leq \varphi_2\leq \cdots,$ be negative plurisubharmonic functions on the unit polydisc $\Delta ^ n \subset \mathbb{ C}^n$ so that $\varphi_n\nearrow \varphi$. Suppose that
$$
\int_{\Delta^n} |F|^2 e^{-\varphi}d\lambda<+\infty,
$$
$F$ is a holomorphic function on $\Delta ^ n$. Then there exists a number $j\geq 1$, such that
$$
\int_{\Delta_r^n}|F|^2 e^{-\varphi_j}d\lambda<+\infty,
$$
for some $r\in (0,1)$.
\end{theorem}

We remark that in this result, it suffices to assume that the $\varphi_j$ are locally uniformly bounded above, rather than being negative. Also, the result is equivalently true if we add any fixed bounded measurable function to all the $\varphi_j$ and the $\varphi$.  Based on these results, we are able to improve Taylor's theorem as follows. We will prove

\begin{theorem}  \label{th:main2}
Let $\varphi_1\leq \varphi_2\leq \varphi_3\leq \cdots $ be plurisubharmonic functions on $\mathbb{C}^n$. For any $\epsilon > 0$ let $\widetilde{\varphi}_j = \varphi_j +  \epsilon \log(1+\|z\|^2)$ and $\widetilde{\varphi} = \lim\limits_{j\rightarrow+\infty} \widetilde{\varphi }_j$. Then $\bigcup \limits _{j=1}^\infty H(\widetilde{\varphi}_j)$ is dense in $H(\widetilde{\varphi})$.
\end{theorem}

%\begin{conjecture} \label{conj1}
%If $\epsilon = 0$ in the theorem \ref{th:main2}, is there the same density result?
%\end{conjecture}

The most direct generalization of the strong openness conjecture to $\mathbb C^n$ is the following:
\begin{conjecture}\label{conj2}
Let $\varphi_1\le \varphi_2\le \cdots$ be plurisubharmonic functions on $\mathbb C^n$
which are locally uniformly bounded above. Let $\varphi=\lim \limits_{k \rightarrow+\infty}\varphi_k.$ If
$f$ is an entire function so that $\int_{\mathbb C^n}|f|^2 e^{-\varphi}d\lambda <\infty$, then
for all large enough $k,$ $\int_{\mathbb{C}^n} |f|^2e^{-\varphi_k}d\lambda <\infty.$
\end{conjecture}

Nevertheless we have a counterexample for Conjecture \ref{conj2}  in $\mathbb{C}^n$. We will discuss it on the fourth part.

Recently, an approach based on H\"ormander's $L^2$-estimates of $\overline{\partial}$ to the openness conjecture was proposed by Chen \cite{boyongChen-20151}, and we will borrow some techniques of his.

For the sake of convenience we always use $C$ to represent universal constants in this paper.

\section{Weighted $L^2$- estimate for the $\overline{\partial}$-equation}
Let $\Omega \subset \mathbb{C}^n$ be a bounded pseudoconvex domain and let  $\varphi$ be a psh function on $\Omega $. By H\"ormander's $L^2$-existence theorem for the $\overline{\partial}$-equation (see \cite{Hormander66}), we know that for every $\overline{\partial}$-closed $(0,1)$-form $v$ on $\Omega$ with $\int _\Omega  |v|^2e^{-\varphi}d\lambda < +\infty$, there exists a solution $u$ to $\overline{\partial} u =v$ such that
$$
\int _\Omega |u|^2e^{-\varphi}d\lambda \leq C_{n, \rm{diam}(\Omega)} \int _\Omega|v|^2e^{-\varphi}d\lambda.
$$
We say that $u$ is the (unique) $L^2(\Omega, \varphi)$-minimal solution of the $\overline{\partial}$-equation if $u\bot \rm{Ker} \overline{\partial}$ in $L^2(\Omega, \varphi)$, i.e., $u$ has minimal norm $\|\cdot\|$ among all solutions.  \\

In this paper we will  first give the following main estimate by H\"ormander \cite{Hormander66}  (see also \cite{boyongChen-2015}):

\begin{theorem}\label{th:Hormander'sestimate}
Let $\Omega \subset \mathbb{C}^n$ be a pseudoconvex domain.  Let $\varphi$ be a psh function on $\Omega$ satisfying
$$
i\partial\overline{\partial}\varphi \geq \Theta
$$
in the sense of distributions for some continuous positive $(1,1)$ -form $\Theta$ on $\Omega$. For any $\overline{\partial}-$ closed $(0,1)-$ form $v$ with
$$
\int_\Omega|v|^2_\Theta e^{-\varphi}d\lambda < +\infty ,
$$
there exists $u\in L^2(\Omega, \rm{loc})$ such that $\overline{\partial} u = v $ and
$$
\int_{\Omega}|u|^2 e^{-\varphi}d\lambda \leq \int_{\Omega}|v|^2_\Theta e^{-\varphi}d\lambda.
$$
\end{theorem}
Here we use the notation: Let $\Theta = i \sum\limits_{j,k} \Theta_{j,k} dz_j \wedge d\overline{z}_k$ be a continuous positive $(1,1)$-form on $\Omega$, i.e., the matrix $(\Theta_{j,k})$ is positive definite at every point. The pointwise norm of a $(0,1)$-form $v$ with respect to $\Theta$ is defined by
$$
|v|^2_\Theta : = \sum _{j,k} \Theta^{j,k} v_j\overline{v}_k
$$
where $(\Theta^{j,k})$ denotes the matrix inverse to $(\Theta_{j,k})$.

\textbf{Remark.} Actually it was Demailly who first gave the above formulation of H\"ormander's estimate. The first remarkable variation of H\"ormander's estimate is the following, \cite{Donnelly1983},

\begin{theorem}[Donnelly-Fefferman]\label{th:D-Fth}
Let $\Omega \subset \mathbb{C}^n$ be a pseudoconvex domain and $\varphi \in psh(\Omega)$. Suppose $\psi $ is a $C^2$ strictly psh function which satisfies
\begin{eqnarray}\label{eq:assumption}
r i\partial\overline{\partial} \psi \geq i\partial \psi \wedge \overline{\partial} \psi
\end{eqnarray}
for some $r >0$. For each $\overline{\partial}$-closed $(0,1)$-form $v$ there exists a solution $u$ of $\overline{\partial} u = v$ satisfying
\begin{eqnarray}\label{eq:estimate1}
\int _{\Omega} |u|^2e^{-\varphi}d\lambda \leq const_r\int _{\Omega} |v|_{i\partial \overline{\partial} \psi}^2e^{-\varphi}d\lambda,
\end{eqnarray}
provided that the RHS of (\ref{eq:estimate1}) is finite.
\end{theorem}

Later, Berndtsson \cite{Boberndtsson2001} generalized theorem \ref{th:D-Fth} as follows.

\begin{theorem} [Berndtsson] \label{th:Hormander1}
Let $\Omega \subset \mathbb{C}^n$ be a pseudoconvex domain and $\varphi \in psh(\Omega)$.  $\psi$ is a $C^2$ strictly psh function satisfying
$$
r i\partial \overline{\partial} (\varphi+\psi) \geq i\partial \psi \wedge \overline{\partial} \psi
$$
in the sense of distributions for some $0<r<1$. Then for each $\overline{\partial}$-closed $(0,1)$-form $v$, there is a  solution of $\overline{\partial} u =v$ which satisfies
\begin{eqnarray}\label{ineq:estimate1}
\int _{\Omega} |u|^2e^{\psi-\varphi}d\lambda \leq \frac{6}{(1-r)^2} \int _{\Omega} |v|_\Theta^2e^{\psi-\varphi}d\lambda
\end{eqnarray}
for every continuous positive $(1,1)-$ form $\Theta$ with $i\partial \overline{\partial}(\varphi+\psi)\ge \Theta$ in the sense of distributions.
\end{theorem}

For the reader's convenience we include the proof of Theorem \ref{th:Hormander1} here.

\begin{proof} We note at first that it suffices to prove the Theorem in the case the right hand side of (\ref{ineq:estimate1}) is finite.
So we suppose that $v$ is a $\overline{\partial}$-closed $(0,1)$-form on $\Omega$ with
$\int _{\Omega} |v|_\Theta ^2e^{\psi-\varphi}d\lambda<\infty.$

We  exhaust $\Omega$ by a sequence of pseudoconvex domains $\Omega_j \subset \subset \Omega$ with smooth boundaries such that $\overline{\Omega}_j \subset \Omega_{j +1}$, $\Omega =\cup \Omega _j$. For each $j$, we may choose  smooth strictly psh function $\varphi_j$ on $\Omega_{j+1}$ such that $\varphi_j$ decrease monotonically to $\varphi$ on $\Omega$ as $ j \rightarrow  \infty$ and 
$$
i \partial \overline{\partial} (\varphi_j +\psi)  \geq \Theta
$$
on $\overline{\Omega}_j$ for each continuous positive $(1,1)-$ form $\Theta$ with $i\partial \overline{\partial}(\varphi+\psi)\ge \Theta$. To see this, simply note that 
$$
i \partial \overline {\partial} (\varphi \ast \theta _\epsilon) = (i \partial \overline {\partial } \varphi )\ast \theta_\epsilon \rightarrow i \partial \overline {\partial} \varphi
$$
%\begin{eqnarray*}
%&& r i\partial \overline{\partial} (\varphi \ast \theta_\epsilon+ \psi )  \nonumber \\
%&&= r i\partial \overline {\partial } (\varphi \ast \theta_\epsilon) + r i\partial \overline{\partial }\psi  = r i\partial \overline {\partial }  (\varphi + \psi ) \ast \theta_\epsilon \nonumber \\
%&& \geq (i\partial \psi \wedge \overline{\partial} \psi ) \ast \theta_\epsilon \rightarrow i\partial \psi \wedge \overline{\partial} \psi. \nonumber \\
%\end{eqnarray*}
on $\overline{\Omega}_j$ as $\epsilon \rightarrow 0$, where $\theta$ is a standard Friedriches mollifier. It suffices to take $\varphi_j = \varphi \ast \theta_{\epsilon_j} + (|z|^2 +1)/j$ with $\epsilon_j = \min \{d(\Omega_j, \partial \Omega), 1/j\} \ll 1$.

Since $i\partial \overline{\partial} (\varphi_j+\psi) \geq \Theta$, $i\partial \overline{\partial} (\varphi_j+\psi)$ is a continuous positive $(1,1)$-form on $\Omega_j$ and $\psi$ is bounded on $\Omega_j$, we obtain that
\begin{eqnarray}
 \int_{\Omega_j}|v|^2_{i\partial \overline{\partial} (\varphi_j+\psi)} e^{-(\psi+\varphi_j)}d\lambda  \leq c_\psi\int_{\Omega_j}|v|^2_\Theta e^{\psi-\varphi}d\lambda<\infty
\end{eqnarray} \label{ineq:1}
since $|v|^2_{i\partial \overline{\partial} (\varphi_j+\psi)} \leq |v|^2_\Theta$. Therefore  Theorem \ref{th:Hormander'sestimate} applies 
with $\varphi_j+\psi$ instead of $\varphi.$ Hence there is a solution to $\overline\partial u_j=v$ on
$\Omega_j$ so that 
\begin{eqnarray}\label{ineq:2}
\int_{\Omega_j} |u_j|^2e^{-\psi-\varphi_j} d \lambda <\infty.
\end{eqnarray}
Since $\psi$ is bounded, $\int_{\Omega_j} |u_j|^2e^{-\varphi_j} d\lambda <\infty.$ We will assume that $u_j$ is chosen as the
unique solution which is perpendicular to the holomorphic functions, i.e.
$\int_{\Omega_j} u_j \cdot \overline{g}e^{-\varphi_j} d\lambda=0$ for all holomorphic functions on $\Omega$
with $\int_{\Omega_j}|g|^2e^{-\varphi_j}d\lambda<\infty.$ 

Because of the boundedness of $\psi$ we have $L^2(\Omega_j, \varphi_j) = L^2(\Omega_j, \varphi_j+\psi)$, and $u_je^\psi  \bot \rm{Ker }     \overline{\partial}$ in $L^2(\Omega_j, \varphi_j+ \psi ) $, i.e.,
$$
\int _{\Omega_j }u_je^\psi \cdot \overline{g} e^{-\psi-\varphi_j}d\lambda  =0
$$
for any such $g$ in $L^2(\Omega_j, \varphi_j+\psi)$.
Thus $u_je^\psi$ is the $L^2(\Omega_j,  \varphi_j+\psi)$-minimal solution of the equation
$$
\overline{\partial} \widetilde{u} = \overline{\partial} (u_je^\psi ).
$$
It follows from Theorem \ref{th:Hormander'sestimate} that
\bea \label{eq: dbarestimate}
\int_{\Omega_j} | u_j|^2e^{ \psi - \varphi_j }d\lambda &= & \int_{\Omega_j} | u_j e^\psi |^2e^{-\psi - \varphi_j }d\lambda \nonumber \\
& \leq & \int_{\Omega_j}| \overline{\partial} (u_je^ \psi)|^2_{i\partial \overline{\partial} (\varphi_j+\psi)} e^{ -\psi - \varphi_j}d\lambda\nonumber \\
&=& \int_{\Omega_j} |v+ \overline{\partial} \psi \wedge u_j|_{i\partial \overline{\partial} (\varphi_j+\psi)}^2 e^{ \psi - \varphi_j }d\lambda\nonumber \\
%& \le &  \int_{\Omega}|v|^2_{i\partial\overline{\partial} (\varphi + \psi)} e^{ \psi - \varphi } d\lambda +
 %\int_\Omega |u|^2 |\overline{\partial} \psi|^2_ {i \partial\overline{\partial} (\varphi + \psi)}e^{ \psi - \varphi %}d\lambda \nonumber  \nonumber \\
% &&+  \int_{\Omega} |v \wedge \overline{\partial} \psi \wedge u| _{ i \partial\overline{\partial} (\varphi + \psi)}^2 e^{ %\psi - \varphi }d\lambda \\
& \le & (1+ \frac{1}{t}) \int_{\Omega_j}  |v|^2 _{i\partial \overline{\partial} (\varphi_j+\psi)} e^{ \psi - \varphi_j } d\lambda \\ && + (1+t) \int_{\Omega_j} |u_j|^2 |\overline{\partial} \psi|^2_ {i\partial \overline{\partial} (\varphi_j+\psi)} e^{ \psi - \varphi_j }d\lambda \\
& \le & (1+ \frac{1}{t}) \int_{\Omega_j}  |v|^2 _\Theta e^{ \psi - \varphi_j }d\lambda \\  && 
+ (1+t) \int_{\Omega_j} |u_j|^2 |\overline{\partial} \psi|^2_ {\frac{1}{r} i\partial \psi \wedge \overline{\partial}\psi} e^{ \psi - \varphi_j}d\lambda,\nonumber \\
& \le & (1+ \frac{1}{t}) \int_{\Omega_j}  |v|^2 _\Theta e^{ \psi - \varphi_j} d\lambda \\ &&+ (1+t)r \int_{\Omega_j} |u_j|^2 e^{ \psi - \varphi_j}d\lambda,\nonumber \\
\eea
in view of
$$
 i\partial \overline{\partial} (\varphi_j+\psi) \geq \frac{1}{r}i\partial \psi \wedge \overline{\partial}\psi
$$
where $0<t<1$.  Since
$$\int_{\Omega_j} |u_j|^2e^{ \psi - \varphi_j} d\lambda\le
const_\psi\int_{\Omega_j} |u_j|^2e^{- \psi - \varphi_j} d\lambda<\infty$$
by (\ref{ineq:2}), so we have
\begin{eqnarray}  \label{ineq:weightedestimate}
\int_{\Omega_j} |u_j|^2e^{ \psi - \varphi_j }d\lambda \le \frac{1+\frac{1}{t}}{1-(1+t) r } \int_{\Omega_j} |v |_ \Theta ^2e^{ \psi - \varphi_j }d\lambda
\end{eqnarray}
provided $(1+t)r< 1$. It is not difficult to see that the cofficient in (\ref{ineq:weightedestimate}) attains the minimum $\frac{1}{(1-  \sqrt{r})^2}$ when $t= \frac{1}{\sqrt{r}}- 1 $.

Since $\{u_j\}$ is uniformly $L^2$ on each compact set of $\Omega$, so we may choose a sequence by the standard diagonal sequence argument, which is still denoted by $\{u_j\}$ for the sake of simplicity, such that $u_j \rightarrow u$ weakly in $L^2(\Omega,\rm{loc})$. For each fixed $k$, according to one corollary of weak convergence theorem we have 
\bea
\int_{\Omega_k} |u|^2 e^{\psi-\varphi_k} d \lambda & \leq  & \liminf \limits_{j \rightarrow + \infty}\int_{\Omega_k}  |u_j|^2 e^{\psi - \varphi_k} \\
& \leq & \liminf \limits_{j \rightarrow + \infty}\int_{\Omega_j}  |u_j|^2 e^{\psi - \varphi_j}  \\
&\leq & \liminf \limits_{j \rightarrow + \infty} \frac{6}{(1-r)^2}\int_{\Omega_j} |v|_\Theta ^2 e^{\psi - \varphi_j}  d\lambda \\
& \leq &   \liminf \limits_{j \rightarrow + \infty}  \frac{6}{(1-r)^2} \int_\Omega |v|_\Theta ^2 e^{\psi - \varphi}  d\lambda \\
&=&  \frac{6}{(1-r)^2}\int_\Omega |v|_\Theta ^2 e^{\psi - \varphi}  d\lambda, \\
\eea
so that 
\bea
\int_\Omega |u|^2 e^{\psi - \varphi} d\lambda  & = &\lim\limits_{k \rightarrow + \infty} \int_\Omega \chi_{\Omega_k}\cdot |u|^2 e^{\psi-\varphi_k} d \lambda =\lim\limits_{k \rightarrow + \infty} \int_{\Omega_k} |u|^2 e^{\psi-\varphi_k} d \lambda \\ 
& \leq  &\frac{6}{(1-r)^2}\int_\Omega |v|_\Theta e^2 e^{\psi - \varphi}  d\lambda\\
\eea
in view of the Lebesgue monotone convergence theorem. The proof is complete.

\end{proof}

\section{Proof of Theorem  \ref{th:main2}}

Our proof depends on $L^2$-theory for the $\overline{\partial}$-operator.

\begin{proof}[Proof of theorem \ref{th:main2}]
Notice that we can replace $\log(1+\|z\|^2)$ by $\log(e+\|z\|^2)$ without changing the spaces and the norms because of equivalence. 

  Here we just give the proof when $0<\varepsilon \leq 1 $. The general case follows by
first adding $(\varepsilon-1)\log (e+\|z\|^2)$ to all $\varphi_j$ and $\varphi.$
Let $\chi : \mathbb{R} \rightarrow [0,1]$ be a smooth function on $\mathbb{C}^n$ satisfying $\chi| _{(-\infty, \frac{1}{2})} = 1$, $\chi| _{(1, +\infty)} = 0$ and $|\chi'| \leq 3$.

Set
$$
\psi = - \log(\log(e+\|z\|^2)).
$$
Then we have
$$
i \partial\overline{\partial} \psi = - i\frac{\partial\overline{\partial} \log(e+\|z\|^2)}{\log(e+\|z\|^2)} + i\frac{ \partial \log(e+\|z\|^2) \wedge \overline{\partial}\log(e+\|z\|^2) }{  (\log(e+\|z\|^2))^2}
$$
and
$$
i \partial \psi\wedge  \overline{\partial}\psi=  i\frac{ \partial \log(e+\|z\|^2) \wedge \overline{\partial}\log(e+\|z\|^2)}{(\log(e+\|z\|^2))^2}.
$$
Thus
\begin{eqnarray}
&& \frac{1}{2}  i \partial\overline{\partial} (\widetilde{\varphi}_j + \epsilon \psi) -  i \epsilon \partial \frac{\psi}{2} \wedge\epsilon \overline{\partial}\frac{\psi}{2} \nonumber \\
&=&  \frac{1}{2}  i \partial\overline{\partial} \varphi_j +   \frac{1}{2}  i \epsilon \partial\overline{\partial} \log(e+\|z\|^2) +  \frac{1}{2}i\epsilon \partial\overline{\partial} \psi- \frac{1}{4} i \epsilon ^2 \partial \psi \wedge \overline{\partial} \psi \nonumber \\
&\geq & \frac{1}{2} \epsilon i  \partial\overline{\partial} \log(e+\|z\|^2) -\frac{1}{2}i \epsilon \frac{  \partial\overline{\partial} \log(e+\|z\|^2)}{\log(e+\|z\|^2)} \nonumber \\
&&+ \frac{1}{2}\epsilon  i\frac{ \partial \log(e+\|z\|^2) \wedge \overline{\partial}\log(e+\|z\|^2) }{  (\log(e+\|z\|^2))^2} \nonumber \\
&& - \frac{1}{4}i  \epsilon^2  \frac{ \partial \log(e+\|z\|^2) \wedge \overline{\partial}\log(e+\|z\|^2)}{(\log(e+\|z\|^2))^2}   \nonumber\\
&\geq &\frac{1}{2}  \epsilon \left(1- \frac{1}{ \log(e+\|z\|^2)}\right)i \partial\overline{\partial} \log(e+\|z\|^2)   \nonumber \\
&&  + \frac{1}{2}   \epsilon  (1- \frac{1}{2} \epsilon  ) i  \frac{ \partial \log(e+\|z\|^2) \wedge \overline{\partial}\log(e+\|z\|^2)}{(\log(e+\|z\|^2))^2}  \nonumber \\
&\geq & 0 \nonumber
\end{eqnarray}

Let $f\in H( \widetilde{\varphi} )$, $ N \in \mathbb N$ so that $\frac{1}{N}<\varepsilon.$
Observe that
$$
\{\frac{N}{2\varepsilon}\le -\psi \le \frac{N}{\varepsilon} \}=
\{-\log 2 \le \log(-\varepsilon \psi)+\log \frac{1}{N}\le 0\}=A.
$$
 Then $f\cdot\chi\left( \log (-\epsilon \psi)+ \log \frac{1}{N}   \right)$ on $A$ extends as a smooth function by
setting it equal to $f$ when  $\log(-\varepsilon \psi)+\log \frac{1}{N}<-\log 2$ and $0$ when
$ \log(-\varepsilon \psi)+\log \frac{1}{N}>0.$ The extension is a smooth function but it is not holomorphic on $\mathbb{C}^n$, so we modify it first.  Put
$$
v_N : = f \cdot \overline{\partial}\chi\left( \log (-\epsilon \psi)+ \log \frac{1}{N}   \right).
$$
Apply Theorem \ref{th:Hormander1} especially for $\Omega = \mathbb{C}^n$ with $\varphi$ and $\psi$ replaced by $\widetilde{\varphi}_j +\frac{\epsilon \psi}{2}$ and $\frac{\epsilon \psi}{2}$ respectively, $r=\frac{1}{2}$ and $\Theta = i \epsilon \partial \psi \wedge\epsilon \overline{\partial}\psi$, we then obtain a solution $u_{j,N}$ of $\overline{\partial} u= v_N$ on $\mathbb{C}^n$ satisfying
\begin{eqnarray*}
&&  \int_{\mathbb{C}^n} |u_{j,N}|^2 e^{ \frac{\epsilon \psi}{2} -(\widetilde{\varphi} _j+ \frac{\epsilon \psi}{2}) }d\lambda \leq \frac{6}{(1-\frac{1}{2})^2}  \int_{\mathbb{C}^n} \left |f \overline{\partial}\chi\right|_\Theta ^2 e^{ \frac{\epsilon\psi}{2} -(\widetilde{\varphi}_j + \frac{\epsilon\psi}{2}) }d\lambda \nonumber \\
& \le & 24 \cdot 9 \int_{\frac{N}{2\epsilon} \leq -  \psi \leq \frac{N}{\epsilon}} |f|^2 \frac{1}{(\epsilon\psi )^2} \left|\epsilon \overline{\partial} \psi \right|_{\epsilon^2 i \partial \psi \wedge  \overline{\partial} \psi }^2 e^{-\widetilde{\varphi }_j}d\lambda\nonumber \\
& \le & \frac{C}{N^2} \int_{\frac{N}{2\epsilon} \leq  -  \psi \leq \frac{N}{\epsilon}}  |f|^2 e^{-\widetilde{\varphi }_j}d\lambda.\nonumber \\
\end{eqnarray*}

Put $K: = \{z: z\in \mathbb{C}^n, \ \ -\psi\le  \frac{N}{ \epsilon} \} $.
Let $q\in K$ and let $\Delta(q)$ be a polydisc centered at $q.$ Since $f\in H(\tilde{\varphi}),$ we have that $\int_{\Delta(q)}|f|^2 e^{-\tilde{\varphi}}d\lambda<\infty.$
By Theorem \ref{th:opennessconjecture} there exists a $j_q\ge 1$ and a number $0<r_q<1$ so that $\int_{\Delta(q)_{r_q}}|f|^2 e^{-\tilde{\varphi}_{j_q}}d\lambda<\infty.$

By compactness there are finitely many $q_i\in K, j_i\in \mathbb N, 1\le i \le m$ so that
$K\subset \cup_{i=1}^m \Delta_{r_i}(q_i)$ and $\int_{\Delta_{r_i}(q_i)} |f|^2 e^{-\tilde{\varphi}_{j_i}}d\lambda<\infty.$ Let $j_0=\max \{j_i\}.$ Then $\int_K |f|^2 e^{-\tilde{\varphi}_{j_0}}d\lambda<\infty.$

But $-\tilde{\varphi}_j$ monotonically decreases to $-\tilde{\varphi}$,  $|f|^2 e^{-\tilde{\varphi}_{j_0}}$ can be seen as the control function on $K$. So we have for all large enough $j\geq j_0$,
$$
 \int_K |f|^2 e^{-\tilde{\varphi_j}}d\lambda \longrightarrow   \int_K |f|^2 e^{-\tilde{\varphi}} d\lambda
$$
in view of the Lebesgue dominated convergence theorem.
Set
$$
F_{j,N} = f \cdot\chi\left( \log (-\epsilon \psi)+ \log \frac{1}{N}   \right) - u_{j,N}.
$$
We then have $F_{j,N} \in \mathcal{O}(\mathbb{C}^n)$ such that for each $j \geq j_0 \gg1$,
\begin{eqnarray*}
&& \|F_{j,N}\|_{H(\widetilde{\varphi}_j)} \leq \left\| f \cdot\chi\left( \log (-\psi)+ \log \frac{1}{N}   \right)\right\|_{H(\widetilde{\varphi}_j)} + \|u_{j,N}\|_{H(\widetilde{\varphi}_j)}\nonumber \\
&\leq & \left(\int _{-\psi < \frac{N}{\epsilon} } |f|^2 e^{-\widetilde{\varphi}_j}d\lambda\right)^{\frac{1}{2}} + \left( \int_{\mathbb{C}^n} |u_{j,N}|^2 e^{-\widetilde{\varphi}_j}d\lambda\right)^{\frac{1}{2}}\nonumber \\
& \leq & (1+ \frac{ C }{N})\|f\|_{H(\widetilde{\varphi})}  < +\infty.
\end{eqnarray*}
Thus $F_{j,N }\in \bigcup\limits_{j=1}^\infty H(\widetilde{\varphi}_j)$. \\
On the other hand, we have
\begin{eqnarray*} \label{eq:estimate12}
&& \int_{\mathbb{C}^n} | F_{j,N}- f|^2 e^{-\widetilde{\varphi}} d\lambda \nonumber \\
& = &  \int_{\mathbb{C}^n} \left| f \cdot\chi\left( \log (-\epsilon \psi)+ \log \frac{1}{N}   \right) - u_{j,N}- f\right|^2 e^{-\widetilde{\varphi}} d\lambda \nonumber \\
&\leq & 2 \int_{\mathbb{C}^n} \left| f \cdot\chi\left( \log (-\epsilon \psi)+ \log \frac{1}{N}   \right)- f\right|^2 e^{-\widetilde{\varphi}} d\lambda + 2 \int_{\mathbb{C}^n} | u_{j,N} |^2 e^{-\widetilde{\varphi}_j} d\lambda \nonumber \\
& \leq &  2 \int_{-\psi \geq \frac{N}{2\epsilon}} |f|^2 e^{-\widetilde{\varphi}}d \lambda
+  \frac{C}{N^2}  \|f\|^2_{H(\widetilde{\varphi})}     \longrightarrow 0    \  \  \  \ (N\longrightarrow+\infty) \nonumber \\
\end{eqnarray*}
Thus $\bigcup \limits _{j=1}^\infty H(\widetilde{\varphi}_j)$ is dense in $H(\widetilde{\varphi})$, which completes the proof.

\end{proof}

\section{ A counterexample for strong openness conjecture to $\mathbb{C}^n$ }

We show that Conjecture \ref{conj2} is however false.

Use the notation $|z|^{2|\alpha|}=|z_1|^{2\alpha_1}\cdots |z_n|^{2\alpha_n}$.
\begin{lemma}
Let $z^\alpha$ be a holomorphic mononomial. The following are equivalent in $\mathbb C^n.$\\
1. $\int_{\|z\|>1} \frac{|z|^{2|\alpha|}}{\|z\|^{N}}d\lambda <\infty;$\\
2. $N\ge 2|\alpha|+2n+1$.\\
\end{lemma}

\begin{proof}
We prove that 1 implies 2.\\
For $k=1,2,\dots,$ let $A_k=\{z; 2^k\le |z_i|\le 2^{k+1}, \forall i=1,\dots,n.\}.$
On $A_k$ we have:\\
a) $|z|^{2|\alpha|}\ge (2^k)^{2|\alpha|}$;\\
b) $\|z\|^{N}\le (\sqrt{n}(2^{k+1}))^N$;\\
c) The volume $|A_k|\ge (\pi (2^k)^2)^{n}$.

Hence we get
\bea
\infty & > & \int_{\|z\|>1} \frac{|z|^{2|\alpha|}}{\|z\|^{N}}d\lambda\\
& > & \sum_k \int_{A_k} \frac{|z|^{2|\alpha|}}{\|z\|^{N}}d\lambda \\
& > & \sum_k \frac{(2^k)^{2|\alpha|}(\pi (2^k)^2)^{n}}{(\sqrt{n}(2^{k+1}))^N}\\
& = & c_N\sum_k 2^{k(2|\alpha|+2n-N)} \\
\eea
This implies that $2|\alpha|+2n-N\leq -1$. Hence 2 follows.

Next we assume 2. Observe that $|z|^{2|\alpha|}\le \|z\|^{2|\alpha|}.$
We then get

$\int_{\|z\|>1} \frac{|z|^{2|\alpha|}}{\|z\|^{N}} d\lambda \le \int_{\|z\|\ge 1}\frac{1}{ \|z\|^{2n+1}}d\lambda <\infty.$
\end{proof}

\begin{lemma}
Let $\psi$ be a plurisubharmonic function on $\mathbb C^n$. Suppose that
there exist positive number $R,C, N$ so that $\psi(z)=C+N\log \|z\|$ for all $z,\|z\|\ge R.$
Then if $\int_{\mathbb{C}^n} |f|^2e^{-\psi}d\lambda<\infty,$ then $f$ is a polynomial of degree at most $\frac{N-2n-1}{2}$.
\end{lemma}

\begin{proof}
We write $f$ in power series, $f=\sum_\alpha a_\alpha z^\alpha$. We have that
$$
\int_{\|z\| \ge R} \sum_{\alpha}|a_\alpha|^2 |z|^{2\alpha} \frac{e^{-C}}{\|z\|^{N}}d\lambda \le \infty.
$$
In particular each individual term in the integral must be finite. By the above Lemma the only non zero terms can be those for which $N\ge 2|\alpha|+2n+1,$ i.e. $|\alpha|\le \frac{N-2n-1}{2}.$
\end{proof}

For $k=1,2,\dots$, let $N_k$ be an increasing sequence so that $N_k\geq 2n+1+2k.$
We define inductively continuous plurisubharmonic functions, $\varphi_1\le \varphi_2\le \cdots$ and $\varphi_j\rightarrow \varphi.$

\bea
\varphi_1 & = & \max\{1, \ln \|z\|\}\\
\varphi_2 & = & \varphi_1, \|z\|\le 2\\
\varphi_2 & = & C_2+N_2 \log \|z\|, \|z\|\ge 2\\
\cdots & & \\
\varphi_{k+1} & = & \varphi_k, \|z\|\le k+1\\
\varphi_{k+1} & = & C_{k+1}+N_{k+1} \log \|z\|, \|z\|\ge k+1\\
\eea

We see that
\begin{lemma}
The limit $\varphi=\lim\limits_k (C_{k}+N_{k}\log \|z\|), k\leq \|z\|\le k+1, k=2,\dots$, $\varphi=\max \{1,\log \|z\|\}$ when
$\|z\|\le 2.$
\end{lemma}

Let $A_\ell:= \int_{\|z\|\le \ell} e^{-\varphi_1}|z_1|^{2\ell}d\lambda$ and set
$B_\ell:= \int_{\|z\|\ge \ell} |z_1|^{2\ell}e^{-\varphi_\ell}d\lambda.$
Pick $\epsilon_\ell$ so small that $\epsilon_\ell (A_\ell+B_\ell)<1/2^{\ell}.$

\begin{lemma}
The function $f=\sum_{k\ge 1} \epsilon_k z_1^k$ belongs to $H(\varphi)$ but not to any
$H(\varphi_k).$
\end{lemma}

\begin{proof}
We only need to show that $\int_{\mathbb C^n} |f|^2 e^{-\varphi} d\lambda<\infty.$
So we need to show that
$\sum_k \int_{\mathbb C^n} \epsilon_k |z_1|^{2k}e^{-\varphi} d\lambda<\infty.$
We get

\bea
\int_{\mathbb C^n} \epsilon_k |z_1|^{2k}e^{-\varphi} d\lambda
& = & \int_{\|z\|\le k} \epsilon_k |z_1|^{2k}e^{-\varphi} d\lambda+\int_{\|z\|\geq k} \epsilon_k |z_1|^{2k}e^{-\varphi} d\lambda\\
& \leq & \epsilon_k(A_k+B_k)\\
& \leq & 1/2^k.\\
\eea
\end{proof}

\textbf{Acknowledgements}  Part of the work was done during the visit of both authors at Tsinghua Sanya International Mathematics Forum in Sanya.  They would like to thank Tsinghua Sanya International Mathematics Forum  for providing a stimulating environment. The second author also thanks Professor Bo-Yong Chen for valuable comments. The first author was supported in part by the Norwegian Research Council grant number 240569 and NSF grant DMS1006294

\end{document}